\documentclass[11 pt]{amsart}
\usepackage[utf8]{inputenc}
\usepackage[margin=2.8cm]{geometry}

\usepackage{amsmath}
\usepackage{mathtools}
\usepackage{amsthm}
\usepackage{amssymb}
\usepackage{graphicx}
\usepackage{verbatim}
\usepackage{mathtools}
\usepackage{tikz}
\usetikzlibrary{graphs}
\usepackage{comment}

\usepackage[colorinlistoftodos]{todonotes}
\usepackage[colorlinks=true, allcolors=blue]{hyperref}
\usepackage{tikz-cd}
\usepackage{comment}

\usepackage[ maxnames=6,minnames=1]{biblatex}
\title{ Fractal behavior of tensor powers of tilting modules of $\SL_2$}
\author{Nai-Heng Sheu}
\address{ Department of Mathematics, Indiana University, Bloomington, IN 47405,
 U.S.A.}
\email{naihsheu@iu.edu}
\theoremstyle{plain}
\newtheorem{thm}{Theorem}[section]
\newtheorem{lemma}[thm]{Lemma}
\newtheorem{prop}[thm]{Proposition}

\theoremstyle{definition}

\theoremstyle{remark}

\newtheorem*{thank}{{\bf Acknowledgments}}
\newtheorem*{notation}{{\bf Notation}}

\def\R{\mathbb{R}}
\def\Q{\mathbb{Q}}
\def\Z{\mathbb{Z}}

\def\N{\mathbb{N}}

\def\SL{\text{SL}}

\def\xx{\otimes}
\def\l{\lambda}

\def\ch{\text{characteristic}}

\def\<{\langle}
\def\>{\rangle}

\def\exp{\text{exp }}

\def\b{\beta}
\def\m{\mu}

\def\supp{\text{supp}}
\def\ext{\text{ext}}

\begin{document}

\begin{abstract}
Given a group $G$ and $V$ a representation of $G$, denote the number of indecomposable summands of $V^{\xx k}$ by $b_k^{G, V}$. 
   Given a tilting representation $T$ of $\SL_2(K)$ where $K=\overline{K}$ and of $\ch$ $p>2$, we show that $Ck^{-\alpha_p}(\dim T)^k<b_k^{T, \SL_2(K)}<Dk^{-\alpha_p}(\dim T)^k$ for some $C, D>0$ where $\alpha_p=1-(1/2)\log_p(\frac{p+1}{2}).$ 
\end{abstract}

\maketitle

\section{Introduction}

Let $G$ be a group and $V$ be a representation of $G$. Decompose $V^{\otimes k}$ into indecomposables, and denote the number of direct summands of $V^{\xx k}$ by $b_k^{G, V}$. In \cite{coulembier2024growth}, Coulembier, Ostrik and Tubbenhauer showed $\lim_{k \to \infty} (b_k^{G, V})^{1/k}=\dim V.$ In \cite{coulembier2024fractalbehaviortensorpowers}, Coulembier, Etingof, Ostrik and Tubbenhauer proposed an asymptotic form $b_k^{G, V}$  
\begin{equation*}
    b^{G, V}_k \sim  h(k)k^{-\alpha} (\dim V)^k
\end{equation*}
where $h: \Z_{\ge 0} \to \R$ is bounded away from $0$ and $\infty$, and $\alpha \in \R_{\ge 0}$. 

In \cite{larsen2025tensorpowerasymptoticslinearly}, Larsen showed that when $G$ is linearly reductive over $K=\overline{K}$ of any characteristic and $V$ a faithful representation of $G$, the number $\alpha$ for $b_k^{G, V}$ is $u/2$ where $u$ is the dimension of the maximal unipotent subgroup of $G$. When $\alpha$ is transcendental or at least irrational, this is described as fractal behavior in \cite{coulembier2024fractalbehaviortensorpowers}.  We will soon see that when $G=\SL_2(K)$ where $K=\overline{K}$ is of positive characteristic, and $T$ is any tilting module of $G$, the corresponding $\alpha$ for $b_k^{G, T}$ are transcendental.

%In this paper, we study the asymptotic behavior of $b_k^{G, V}$ in the case $G=\SL_2(K)$ where $K$ is algebraically closed, and of $\ch$ $p>2$ and $V$ is a tilting representation of $G$.

%When $K$ is of $\ch$ zero, we may assume $K=\C$ and study $b_k^{G, V}$ via the group $\text{SU}_2(\C)$ since the representation category of $\SL_2(\C)$ is monoidally equivalent to the representation category of $\text{SU}_2(\C)$, and we can integrate over $\text{SU}_2(\C)$. Let $V$ be a representation of $\text{SU}_2(\C)$. The number of summands of $V^{\xx n}$ is more than the number of trivial summands in $V^{\xx n}$, which is the integral of $\chi_{V^{\xx n}}$ and is approximated by the integral over an open neighborhood of the identity element. Therefore, the number of summands of $V^{\xx n}$ is approximated by $\dim V^{\xx n}=(\dim V)^n.$

%When $K$ is of positive \ch, and $V$ a representation of $\SL_2(K)$, as there is no compact form (?) of $\SL_2(K)$, we can't understand the asymptotic behavior of the number of summands of $V^{\xx n}$ in the same way. As any $\SL_2(\C)$-representation is tilting, it suggests understanding the asymptotic behavior when $V$ is a tilting representation of $\SL_2(K)$ where $K$ algebraically closed and of positive \ch. This topic has been studied in \cite{coulembier2024fractalbehaviortensorpowers}, and in \cite{larsen2024boundsmathrmsl2indecomposablestensorpowers} where each of them approached this question in different directions. 

When $G=\SL_2(K)$ where $K$ is algebraically closed and of $\ch$ $p>0$, and $V$ is the natural $2$-dimensional representation of $G$, the growth of $b_k^{G, V}$ has been studied in \cite{coulembier2024fractalbehaviortensorpowers} and \cite{larsen2024boundsmathrmsl2indecomposablestensorpowers} from two rather different directions. 

In \cite{coulembier2024fractalbehaviortensorpowers}, Coulembier, Etingof, Ostrik and
Tubbenhauer gave \begin{equation} \label{formofbounds}
    Ck^{-\alpha_p}2^k < b_k^{G, V}< D k^{-\alpha_p}2^k
\end{equation}
where $\alpha_p=1-(1/2)\log_p(\frac{p+1}{2}),$ which is a transcendental number, and proposed that $\alpha_p$ could be the fractal dimension of some fractal.

In \cite{larsen2024boundsmathrmsl2indecomposablestensorpowers}, Larsen studied the case $p=2$, showed that $$b_{2k}^{G, V}\sim \omega(k)k^{-\alpha_2} 4^k$$ where  $\omega(x)$ is multiplicative periodic function,  and gave a lower bound of $b_k^{G, T}$ where $T$ is any tilting representation of $G$.

In this paper, we adopt the method in \cite{larsen2024boundsmathrmsl2indecomposablestensorpowers}, and use the  results from \cite{coulembier2024fractalbehaviortensorpowers}, to generalize  the results in \cite{coulembier2024fractalbehaviortensorpowers} and \cite{larsen2024boundsmathrmsl2indecomposablestensorpowers} and prove the following theorem.

\begin{thm}
    Let $T$ be a tilting representation of $\SL_2(K)$ where $K$ is algebraically closed and of $\ch$ $p>2$. Then there exist some positive numbers $C_T$ and $D_T$ such that when $k$ is sufficiently large, 

    \begin{equation} \label{main_ineq}
        C_Tk^{-\alpha_p} (\dim T)^k < b_k^{G, T}< D_T k^{-\alpha_p}(\dim T)^k    
    \end{equation}
    where $\alpha_p$ is as in \eqref{formofbounds}.     
\end{thm}

An indecomposable tilting module of $G$ is characterized by its highest weight. Let $T(n)$ denote the indecomposable tilting module of highest weight $n$. Define the generating function $$X_n(t)=\sum_{i\ge 0}a_{n, i}t^i$$ where $a_{n, i}$ is the multiplicity of $T(n)$ in $V^{\xx i}$. Then $b_k^{G, V}=\sum_{0 \le n \le k}a_{n, k}$.  We use a partial fraction expansion to understand the asymptotic behavior of $X_n(t)$, which enables us to use the equation \eqref{M_k} to reduce the approximation of $b_k^{G, T}$ to that of $b_l^{G, V}$ where $l$ is linear in $k$. We show that the generating functions $X_n(t)$ have two striking properties which are not obvious a priori. One is that $X_n(t)$ are rational functions, and the other is that  $X_n(t)$ are multiplicative in terms of base $p$ expansion. These properties lead us to use the partial fraction expansion on $X_n(t)$.

%To each indecomposable tilting representation, a generating function $X_n$ is associated. After an index shift, the generating functions $X_n$ are multiplicative with respect to the base $p$ expansion of $n$, and $X_{ap^s}$ is defined inductively from $X_a$. This suggests a connection of tilting modules with Steinberg's tensor product Theorem.

The paper is structured as follows. First, we fix notations and associate a generating function to each indecomposable tilting module.  In the second section, we recall the explicit decomposition of $T\xx V$ where $T$ is an indecomposable tilting module and $V$ is the natural representation of $\SL_2(K)$, and use this decomposition to obtain the properties of the generating functions. In the end, we use the estimations of the coefficients of the generating functions to give an upper bound and a lower bound of $b_k^{G, T}$ when $k$ is large.

\begin{thank}
    I would like to thank Michael Larsen for all of the comments, discussions and support, and Matthias Strauch for the encouragement. 
\end{thank}

\section{notation}

Fix a prime $p>2$. Let $K$ be an algebraically closed field of characteristic $p$.  Let $G=\SL_2(K)$ and $V$ be the natural  $2$-dimensional representation of $G$.  

Given a non-negative integer $n$, denote the tilting module of highest weight $n$ by $T(n)$ and the character of $T(n)$ by $\chi_n(t)$. Then $V \cong T(1).$ 

Let $X_n=\sum_{i=0}^\infty a_i t^i$ be the generating function associated to $T(n)$ where the coefficient $a_i$ is the multiplicity of $T(n)$ in $V^{\xx i}.$

For each non-negative $n=\sum_{i=0}^j a_ip^i$ where $p>a_i \ge 0$, we write $n=[a_j,\ a_{j-1},\ \ldots,\ a_0].$   

\section{The formula of $X_n$}
%We start with considering the decomposition of $T(n)\xx T(1)$. The fractal pattern of decomposition

%\section{$n\ge p-1$}
\subsection{Decomposition of $T(n)\xx T(1)$}

Fix a prime $p>2$. For any non-negative integer $n$, the character $\chi_n(t)$ of $T(n)$ can be expressed in terms of the base $p$ expansion of $n+1$. Say, $n+1=a_j p^j+ a_{j-1}p^{j-1}+\cdots +a_0$ where $0 \le a_i<p$. 
Define the set $$\supp (n)= \{ a_j p^j \pm a_{j-1} p^{j-1} \pm \cdots \pm a_0\}$$ and the multiset $$\ext(n)=\{ x+1, x-1 \mid x \in \supp(n)\}.$$ 

%Denote the $p$-expansion of $n+1$ by $[a_j,\ a_{j-1}, \ldots, \  a_0]$.

By \cite[, Proposition 5.4]{tubbenhauer2021quivers}, \begin{equation} \label{char_n}
    \chi_n= \sum_{k \in \supp(n)} \frac{t^k-t^{-k}}{t-t^{-1}}.
\end{equation}

Since $(t^k - t^{-k})(t+t^{-1})=t^{(k+1)}-t^{-(k+1)}+t^{(k-1)}-t^{-(k-1)}$, \[ \chi_n \cdot \chi_1 = \sum_{k \in \ext(n)} \frac{t^k-t^{-k}}{t-t^{-1}}. \]

Define $T(-1)=\{0\}$ so that $\chi_{-1}$ satisfies the character formula.

\begin{prop} \label{decompT(m).T(1)}
Let $n$ be a non-negative integer, and write $n+1=[a_j, a_{j-1}, \ldots, a_0]$. 

Suppose $a_0 \neq p-1$, then  \begin{align*}
    T(n)\xx T(1)=\begin{cases}
        T(n+1) \quad \quad \quad \quad \quad \quad,a_0=0\\
        T(n+1)\oplus 2T(n-1) \ \  ,a_0=1\\
        T(n+1)\oplus T(n-1)\ \ \ , 1<a_0<p-1  \end{cases}
\end{align*}

        Suppose $a_0=p-1$. Let $d$ be the largest integer such that $p^{d+1} \mid n+2$ and $p^{d+2} \nmid n+2$, that is equivalent to saying $n+1=[a_j, \ldots,\   a_{d+1},\  p-1, \ldots, \ p-1]$ and $a_{d+1} \neq p-1$. 
        
        Then 
\begin{align*}
     &T(n)\xx T(1)=T(n+1)\oplus T(n-1) \oplus 
     \bigoplus_{i=1}^d T(n+1-2p^i)\\
     &\oplus \begin{cases}
             \begin{cases}
                0, \qquad \qquad \qquad \qquad \quad \ \ a_{d+1}=0\\
                T(n+1-2p^{d+1}), \quad 1<a_{d+1}<p-1\\
                2T(n+1-2p^{d+1}), \qquad \  a_{d+1}=1
            \end{cases}, n+2\neq ap^{d+1} \text{for some $0<a<p$} \\
 T(n+1-2p^{d+1}), \qquad \qquad \qquad \qquad \qquad \ \  n+2=ap^{d+1}, 1<a<p\\
 0, \qquad \qquad \qquad \qquad \qquad \qquad \qquad \qquad \quad \ \ n+2=p^{d+1} %\text{ or }m+2=2p^{d+1}.     
        \end{cases}
\end{align*}
       \end{prop}

The proposition gives Figure 1 and Figure 2 as subgraphs of the fusion graphs of $\xx V$ of the cases $p=3$ and $p=5$.

\usetikzlibrary{arrows.meta}
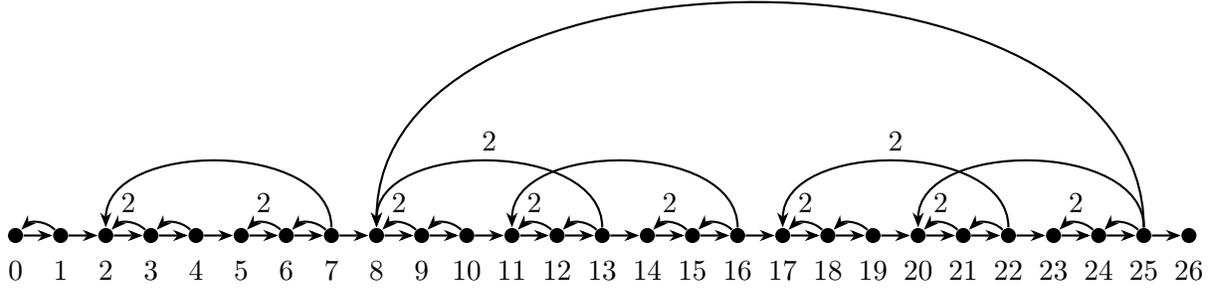
\begin{figure}[h] \label{fig_p=3}
    \centering
    \begin{tikzpicture}[
    node distance=0.6cm,
    arrow/.style={-{Stealth[length=2mm]}, thick},
    curved arrow/.style={-{Stealth[length=2mm]}, thick, bend left=45},
    bullet/.style={circle, fill=black, minimum size=2mm, inner sep=0pt}
]

% Draw nodes 0 to 26
\foreach \i in {0,...,26} {
    \node[bullet] (\i) at (\i*0.6, 0) {};
    \node[below=1mm of \i] {\i};
}

% Horizontal arrows between adjacent nodes
\foreach \i in {0,...,25} {
    \pgfmathtruncatemacro{\next}{\i+1}
    \draw[arrow] (\i.east) -- (\next.west);
}

% Curved arrows from 3n+1 to 3n (1->0, 4->3, 7->6, etc.)
\foreach \n in {0,...,9} {
    \pgfmathtruncatemacro{\from}{3*\n+1}
    \pgfmathtruncatemacro{\to}{3*\n}
    \ifnum\from<27
        \draw[curved arrow] (\from) to[bend right=45] (\to);
    \fi
}

% Curved arrows from 3n to 3n-1 with label "2" (3->2, 6->5, 9->8, etc.)
\foreach \n in {1,...,9} {
    \pgfmathtruncatemacro{\from}{3*\n}
    \pgfmathtruncatemacro{\to}{3*\n-1}
    \ifnum\from<27
        \draw[curved arrow] (\from) to[bend right=45] node[above, midway] {2} (\to);
    \fi
}

% Additional arrows: 7->2 and 16->11
\draw[curved arrow] (7) to[bend right=90] (2);
\draw[curved arrow] (16) to[bend right=90] (11);
\draw[curved arrow] (25) to[bend right=90] (20);

\draw[curved arrow] (13) to[bend right=90] node[above, midway] {2} (8);
\draw[curved arrow] (22) to[bend right=90] node[above, midway] {2} (17);

\draw[curved arrow] (25) to[bend right=90] (8);

\end{tikzpicture}
    \caption{The subgraph for $27> i \ge 0$ when $p=3$}
    %\label{fig:placeholder}
\end{figure}
    \begin{figure}[h] \label{fig_p=5}
    \centering
    \begin{tikzpicture}[
    node distance=0.6cm,
    arrow/.style={-{Stealth[length=2mm]}, thick},
    curved arrow/.style={-{Stealth[length=2mm]}, thick, bend left=45},
    bullet/.style={circle, fill=black, minimum size=2mm, inner sep=0pt},
    scale=0.85
]

% Draw nodes 0 to 34
\foreach \i in {0,...,34} {
    \node[bullet] (\i) at (\i*0.6, 0) {};
    \node[below=1mm of \i] {\i};
}

% Horizontal arrows between adjacent nodes
\foreach \i in {0,...,33} {
    \pgfmathtruncatemacro{\next}{\i+1}
    \draw[arrow] (\i.east) -- (\next.west);
}

% Curved arrows from 5n+1 to 5n (1->0, 6->5, 11->10, etc.)
\foreach \n in {0,...,7} {
    \pgfmathtruncatemacro{\from}{5*\n+1}
    \pgfmathtruncatemacro{\to}{5*\n}
    \ifnum\from<34
        \draw[curved arrow] (\from) to[bend right=45] (\to);
    \fi
}

% Curved arrows from 5n+3 to 5n+2 (3->2, 8->7, 13->12, etc.)
\foreach \n in {0,...,7} {
    \pgfmathtruncatemacro{\from}{5*\n+3}
    \pgfmathtruncatemacro{\to}{5*\n+2}
    \ifnum\from<34
        \draw[curved arrow] (\from) to[bend right=45] (\to);
    \fi
}

% Curved arrows from 5n+2 to 5n+1 (2->1, 7->6, 12->11, etc.)
\foreach \n in {0,...,7} {
    \pgfmathtruncatemacro{\from}{5*\n+2}
    \pgfmathtruncatemacro{\to}{5*\n+1}
    \ifnum\from<34
        \draw[curved arrow] (\from) to[bend right=45] (\to);
    \fi
}

% Curved arrows from 5n to 5n-1 with label "2" (5->4, 10->9, 15->14, etc.)
\foreach \n in {1,...,7} {
    \pgfmathtruncatemacro{\from}{5*\n}
    \pgfmathtruncatemacro{\to}{5*\n-1}
    \ifnum\from<34
        \draw[curved arrow] (\from) to[bend right=45] node[above, midway] {2} (\to);
    \fi
}

% Additional arrows: 13->4, 18->9 and 23->14
\draw[curved arrow] (13) to[bend right=90] (4);
\draw[curved arrow] (18) to[bend right=90] (9);
\draw[curved arrow] (23) to[bend right=90] (14);
\draw[curved arrow] (33) to[bend right=90] node[above, midway] {2} (24);

\end{tikzpicture}
    \caption{The subgraph for $25> i \ge 0$ when $p=5$}
    \label{fig:placeholder}
\end{figure}
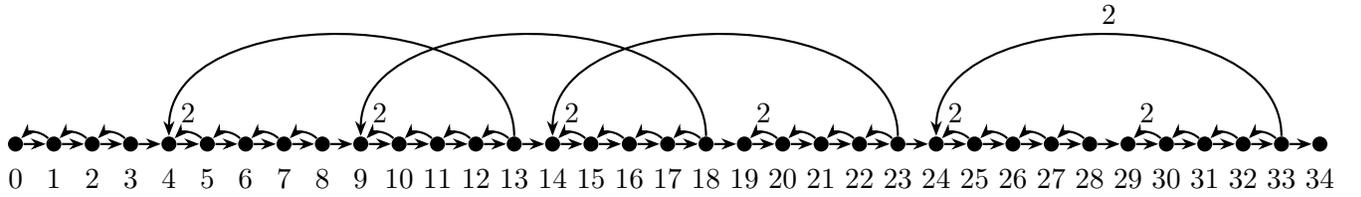

\begin{proof}
Finding the decomposition of $T(n) \xx T(1)$ is equivalent to writing $\ext(n)$ as a disjoint union of $\supp (m)$ for some $m$. If $\ext(n)= \sqcup_{m\in S} \supp(m) $, then $T(n)\xx T(1)=\oplus_{m \in S} T(m).$

Case 1) $a_0=0$:

Since $n+1= [a_j,\ a_{j-1}, \ldots,\ a_1,\ 0]$, we have $\ext(n)=\{ x+1, x-1 \mid x \in \supp(n)\}=\supp(n+1).$ \\

Case 2) $a_0=1$:

$\ext(n)= \supp (n+1) \sqcup_{i=1}^2 \supp(n-1).$\\

Case 3) $1<a_0<p-1$:

$\ext(n)=\supp(n+1) \sqcup \supp(n-1)$.\\

Case 4) $a_0=p-1$:

Write \[n+1=[a_j,\ a_{j-1}, \ldots,\ a_{d+1},\  p-1, \ldots, \ p-1, \ p-1].\] The analysis of $\ext(n)$ splits into the cases $d< j-1$ (i.e. $n+2\neq ap^{d+1}$ for some $0<a<p$), $d=j-1$  (i.e. $n+2=ap^{d+1}$ for some $1<a<p$), and $d=j$ (i.e. $n+2=p^{d+1})$.

Consider the case $d< j-1$,
\begin{align*}
    \ext(n)&= \{ [a_j,\ \pm a_{j-1}, \ldots,\ \pm  a_{d+1},\ \pm (p-1), \ldots,\ \pm (p-1), \ \pm p] \} \\
    &\sqcup \{ [a_j,\ \pm a_{j-1}, \ldots,\ \pm  a_{d+1},\ \pm (p-1), \ldots,\ \pm (p-1), \ \pm (p-2)] \}\\
    &=\{ [a_j,\ \pm a_{j-1}, \ldots,\ \pm  a_{d+1},\ \pm (p-1), \ldots, \ \pm (p-1), \ \pm p] \} \sqcup \supp(n-1).
\end{align*}

The set $\{[a_j,\ \pm a_{j-1}, \ldots,\ \pm  a_{d+1},\ \pm (p-1), \ldots, \ \pm (p-1), \ \pm p] \}$ is

\begin{equation} \label{1}
    \begin{aligned}
    \{  [a_j,\ \pm a_{j-1}, \ldots,\ \pm  &a_{d+1},\ \pm (p-1), \ldots,\ \pm (p-1), \ \mp (p-2), \ 0],\\
      [a_j,\ \pm a_{j-1}, \ldots,\ \pm  &a_{d+1},\ \pm (p-1), \ldots,\ \pm (p-1), \ \mp (p-2), \ 0, \ 0], \\
  &\vdotswithin{a_{d+1}}\\
[a_j,\ \pm a_{j-1}, \ldots,\ \pm & a_{d+1},\ \mp (p-2),\ 0, \ldots,  \ 0, \ 0],\\       
    [a_j,\ \pm a_{j-1}, \ldots,\  (\pm & a_{d+1})\pm 1,\ 0,\ 0, \ldots,  \ 0, \ 0]  
    \},
\end{aligned}
\end{equation}

where the elements in the $k$th row of (\ref{1}) are elements of the set  
\begin{align*}
  \{x+1 \mid x \in \supp (n) \text{ with the $i$-th digit of $x$ is $p-1$ for $0\le i<k$ and the $k$-digit is $-(p-1)$} \} \\
  \sqcup \{x-1 \mid x \in \supp (n) \text{ with the $i$-th digit of $x$ is $-(p-1)$ for $0\le i<k$ and the $k$-digit is $p-1$} \}.
\end{align*}

Therefore, 
\begin{align*}
 &\{[a_j,\ \pm a_{j-1}, \ldots,\ \pm  a_{d+1},\ \pm (p-1), \ldots, \ \pm (p-1), \ \pm p] \} \\
 = &\supp(n+1-2p)\sqcup \supp (n+1-2p^2) \sqcup \cdots \sqcup \supp(n+1-2p^d) \\ &\sqcup \{[a_j,\ \pm a_{j-1}, \ldots,\  (\pm a_{d+1})\pm 1,\ 0,\ 0, \ldots,  \ 0, \ 0]  
    \}. 
\end{align*}
The set 
\begin{align*}
        &\{[a_j,\ \pm a_{j-1}, \ldots,\  (\pm a_{d+1})\pm 1,\ 0,\ 0, \ldots,  \ 0, \ 0]  
    \}\\
    &=\supp(n+1) \sqcup \begin{cases}
         \varnothing,\ \text{when } a_{d+1}=0 \\
        \supp(n+1-2p^{d+1}),\ \text{when } 2\le a_{d+1}<p-1\\
        \sqcup_{i=1}^2 \supp(n+1-2p^{d+1}),\ \text{when } a_{d+1}=1.
    \end{cases}
\end{align*}

A similar argument applies to the case $d=j-1$ where the set of the elements in the last row of ($\ref{1}$) is $$\{[a_{d+1}\pm 1,\ 0, \cdots, \ 0]\}=\supp(n+1) \sqcup \supp (n+1-2p^{d+1}),$$ and the case $d=j$ where the set of the elements in the last row of ($\ref{1}$) is $$\{[(p-1)\pm 1,\ 0, \cdots, \ 0]\}=\supp(n+1) \sqcup \supp(n+1-2p^d).$$

Hence, 
\begin{align*}
    \ext(n)&=\supp(n+1) \sqcup \supp (n-1) \sqcup \sqcup_{i=1}^d \supp(n+1-2p^i)\\
    &\sqcup \begin{cases}
        \begin{cases}
         \varnothing,\ & a_{d+1}=0 \\
        \supp(n+1-2p^{d+1}),\ & 2\le a_{d+1}<p-1\\
        \sqcup_{i=1}^2 \supp(n+1-2p^{d+1}),\ & a_{d+1}=1\end{cases}\ &\text{when } d<j-1,\\    
\supp(n+1-2p^{d+1}),\ &\text{when } d=j-1,\\
\varnothing,\ &\text{when } d=j. \end{cases}
\end{align*}
\end{proof}

\subsection{Multiplicative property of $X_n$}

\begin{lemma} \label{lineareqofX_n}
    Given $n\in \Z_{\ge 0}$, write $n+1=[a_j,\ a_{j-1},\ \ldots,\ a_0].$
    Set $X_{-1}=1/t$. 
    The generating functions $X_n$ satisfy the following linear equations.  

When $a_0 \neq 0$,
\begin{equation} \label{eq3}
X_n=
    \begin{cases}
        t(X_{n-1}+X_{n+1}), \text{ when } 1 \le a_0 < p-1,\\
        t X_{n-1}, \text{when } a_0=p-1.\\
        \end{cases}
\end{equation}

When $a_0=0$,
\begin{equation} \label{eq2}  
\begin{aligned} 
    X_n&=t(X_{n-1}+2X_{n+1}+\sum_{i=1}^d 2X_{n+2p^i-1})\\
    &+\begin{cases}
        0,\ &n+1\equiv (p-1)p^{d+1} 
\quad  \text{(mod $p^{d+2}$)}\\
        tX_{n-1+2p^{d+1}},\ &n+1\equiv ap^{d+1},\ 1\le a \le p-2 \quad \text{(mod $p^{d+2}$)} 
        \end{cases} 
\end{aligned}
   \end{equation}
\end{lemma}

\begin{proof}
    
By the definition of the generating function $X_n$, $$X_n=t(\sum_m c_mX_m),$$ where $c_m$ is defined so that  $T(n)$ is a direct summand of $T(m)\xx T(1)$ with multiplicity $c_m>0$.

Suppose $n+1\ge p$ and write $n+1=[a_j,\ a_{j-1},\ \ldots,\ a_0].$
By Proposition \ref{decompT(m).T(1)}, it follows
\begin{align*}
    X_n&=t(X_{n-1}+X_{n+1}), \text{ when } 1 \le a_0 < p-1, \\
    X_n&=t X_{n-1}, \text{when } a_0=p-1.
\end{align*}

When $a_0=0$, $X_n=t( X_{n-1}+2X_{n+1}+\sum_m c_mX_m)$ where $m+1-2p^k=n$ for some $k$. It suffices to find all such $k$. The proof splits into the case that $n+1=ap^{l+1}$ for some $0<a<p$ and the case that $n+1 \neq ap^{l+1}$ for some $0<a<p$. 

Case 1) $n+1=ap^{l+1}$:  

The base $p$ expansion of $n+1$ is $[a,\ 0, \cdots, \ 0]$ where the leftmost digit is the digit in position $l+1$. By assumption, $$m+1=n+2p^k.$$

Suppose $k \le l$. Then $m+1=[a,\ 0, \cdots, 0,\ 1,\ p-1, \cdots, \ p-1]$ where $1$ is the $k$-th digit. Then $T(n)$ is a direct summand of $T(m)$ with multiplicity $2$ by Proposition $\ref{decompT(m).T(1)}$ with $d=l$.\\ 

Suppose $k=l+1$, $m+1=[a+1,\  p-1, \cdots, \ p-1]$. Then $T(n)$ is a direct summand of $T(m)$ with multiplicity $1$ when $1\le a \le p-2$.\\

Suppose $k>l+1$, $T(n)$ is not a direct summand of $T(m)\xx T(1)$.
\\

Therefore, 
\begin{align*}
    X_n&=t(X_{n-1}+2X_{n+1}+\sum_{i=1}^l 2X_{n+2p^i-1})\\
    &+\begin{cases}
        0,\ n+1=(p-1)p^{l+1}\\
        tX_{n+2p^{l+1}-1},\ n+1=ap^{l+1},\ 1\le a \le p-2.
\end{cases}
\end{align*}

Case 2) $n+1 \neq ap^{l+1}$ for all $0<a<p$:

The assumption gives that the base $p$ expansion of $n+1$ has at least two initial non-zero digits. Say $$n+1=[a_j,\ \cdots, \ a_{d+1},\ 0, \cdots, \ 0],$$ where $d+1<j$ and $a_{d+1}\neq0$. Consider $m+1=n+2p^k$ for some $k$.\\

If $k \le d$, then $m+1=[a_j, \cdots, \ a_{d+1},\ 0, \cdots,\ 0, \ 1,\ p-1, \cdots, \ p-1]$ where $1$ is the $k$-th digit and $T(n)$ is a direct summand of $T(m)\xx T(1)$ with multiplicity $2$.
\\

If $k=d+1$, $m+1=[a_j, \cdots, \ a_{d+1}+1,\ p-1, \cdots, \ p-1],$ and \begin{align*}
    T(m)\xx T(1)=&T(m+1) \oplus T (m-1) \oplus \bigoplus_{i=1}^dT(m+1-2p^i)\\
    &\oplus \begin{cases}
        0, &a_{d+1}+1=0 \Leftrightarrow a_{d+1}=p-1\\
        2T(m+1-2p^{d+1}), &a_{d+1}+1=1 \Leftrightarrow a_{d+1}=0\\
        T(m+1-2p^{d+1}), &1<a_{d+1}+1 \le p-1 \Leftrightarrow 1 \le a_{d+1} \le p-2.
    \end{cases}
\end{align*}
The case $a_{d+1}=0$ is not possible by the assumption on $n+1$.
\\

If $k>d+1$, $T(n)$ is not a direct summand of $T(m)\xx T(1)$.

Therefore, \begin{align*}
    X_n&=t(X_{n-1}+2X_{n+1}+\sum_{i=1}^d 2X_{n+2p^i-1})\\
    &+\begin{cases}
        0,\ &n+1\equiv (p-1)p^{d+1} 
\quad  \text{(mod $p^{d+2}$)}\\
        tX_{n+2p^{d+1}-1},\ &n+1\equiv ap^{d+1},\ 1\le a \le p-2 \quad \text{(mod $p^{d+2}$)}.
\end{cases}
\end{align*}

Combining both cases, one has the equation \eqref{eq2}.
\end{proof}

Using Lemma \ref{lineareqofX_n} and the method to get (2.6) in \cite{larsen2024boundsmathrmsl2indecomposablestensorpowers}, and defining $Z_n=X_{n-1}$, one has the following multiplicative formula 

\begin{prop} \label{ap^s+n=ap^s+n}
    For $0\le a<p$, $\ s \ge 0$, and $0 \le i < p^s$, then
    $$Z_{ap^s+i}=tZ_{ap^s}Z_i.$$    
\end{prop}

\begin{proof}

Given $s\in \N$, for $0 \le n \le p^s-1$, consider the set of linear forms $L_n^s$ in the $p^s+1$ variables $y_i$, $0 \le i \le p^s$ and

\begin{equation*}
     L_n^s(y_{0}, \cdots, y_{p^s-1})
     =\begin{cases}
         t(y_{n-1}+y_{n+1}), \text{ when } \ p \nmid n(n+1),\\
         t y_{n-1}, \text{when } p \mid n+1,\\         
         t(y_{n-1}+2y_{n+1}+\sum_{i=1}^d 2y_{n+2p^i-1})
    \\
    +\begin{cases}
        0,\ n+1=(p-1)p^{d+1}\\
        ty_{n-1+2p^{d+1}},\ n+1=ap^{d+1},\ 1\le a \le p-2, 
        \end{cases} 
        \text{ when }p \mid n.
        \\        
     \end{cases}
 \end{equation*}

Let $L^s(y_{0}, \cdots, y_{p^s-1})$ be the linear system $L_n^s(y_{0}, \cdots, y_{p^s-1})-y_n=0$ for $1 \le n \le p^s-1$. Then $L^s(y_{0}, \cdots, y_{p^s-1})$ gives a $(p^s-1) \times p^s$ matrix $A^s$ of rank $p^s-1$ as the rank of $A^s$ (mod $t\Q(t)$) is of rank $p^s-1$.

As both $(Z_{0},\ Z_0,\cdots,\ Z_{p^{s}-1})$ and $(Z_{ap^{s}},\ Z_{ap^{s}+1},\cdots,\ Z_{(a+1)p^{s}-1})$ are solutions of $L^s$, $$(Z_{ap^{s}},\ Z_{ap^{s}+1},\cdots,\ Z_{(a+1)p^{s}-1})=q(t)(Z_{0},\ Z_1,\cdots,\ Z_{p^{s}-1})$$ for some $q(t) \in \Q(t)$. Since $Z_{ap^{s}}=q(t) Z_{0}$ and $Z_{0}=1/t$, it follows that $q(t)=t Z_{ap^{s}}$ and $$Z_{ap^s+i}=tZ_{ap^s}Z_i$$ for $0\le i< p^s$. 

 \end{proof}

\subsection{The form of $Z_{ap^s}$}
In the following, we work with the generating functions $Z_n=X_{n-1}$, and $Z_0=1/t$.

%\subsubsection{Write $Z_{ap^s}$ as a composition of two functions}

\begin{lemma} \label{Z_p^s relations}
    Given $s \ge0$, define     $c_s= \frac{tZ_{p^s}}{1+tZ_{2p^s}}$. Then \begin{equation} \label{Z_ap^s_back_forth}
    Z_{ap^s}=\begin{cases}
        c_s (Z_{(a-1)p^s}+Z_{(a+1)p^s}),\ 1\le a \le p-2\\
        c_sZ_{(p-2)p^s}, a=p-1.
    \end{cases}        
    \end{equation}
    
\end{lemma}

\begin{proof}
When $s=0$, by the definition, $c_0=tZ_1/(1+tZ_2)$, which is equal to $t$ as $Z_1=X_0=1+tX_1=1+tZ_2=t(1/t+Z_2)$. Therefore, the equation \eqref{Z_ap^s_back_forth}
holds from the decomposition of $T(n)\xx T(1)$ for $p>n \ge 0.$

Suppose $s \ge 1$. 

For $a=1$, the formula holds since $Z_0=1/t$.

For $a>1$, we begin with two observations from the multiplicative property, 

\begin{equation} \label{ob_1}
    Z_{kp^s+n}Z_{lp^s}=Z_{kp^s}Z_{lp^s+n}
\end{equation}
when $0 < n \le p^s$ and 
\begin{equation} \label{ob_2}
    Z_{kp^s-n} Z_{lp^s}=Z_{(k-1)p^s+(p^s-n)} Z_{lp^s}=Z_{(k-1)p^s}Z_{(l+1)p^s-n}
    \end{equation}
    when $1\le k \le p$, $1 \le l <p$ and $0< n \le  p^s$.

For any $a \le p-2$ and $s>0$, by (\ref{eq2}), $$Z_{ap^s}=t(Z_{ap^s-1}+\sum_{i=0}^{s-1}2Z_{ap^s+2p^i-1}+Z_{ap^s+2p^s-1}).$$

Multiplying both sides by $Z_{p^s}$, then \begin{align*}
    Z_{p^s}Z_{ap^s}&=tZ_{p^s}Z_{ap^s-1}+tZ_{p^s}\sum_{i=0}^{s-1}2Z_{ap^s+2p^i-1}+tZ_{p^s}Z_{ap^s+2p^s-1}\\
    &\stackrel{(i)}{=}  tZ_{p^s}Z_{ap^s-1}+tZ_{ap^s}\sum_{i=0}^{s-1}2Z_{p^s+2p^i-1}+tZ_{(a+1)p^s}Z_{2p^s-1}\\ 
    &\stackrel{(ii)}{=}  tZ_{p^s}Z_{ap^s-1}+Z_{ap^s}(Z_{p^s}-tZ_{p^s-1}-tZ_{3p^s-1})+tZ_{(a+1)p^s}Z_{2p^s-1},
\end{align*}

where $(i)$ comes from \eqref{ob_1} and \eqref{ob_2}, and $(ii)$ comes from $$Z_{p^s}=t(Z_{p^s-1}+\sum_{i=0}^{s-1} 2Z_{p^s+2p^i-1}+Z_{p^s+2p^s-1}).$$

Subtracting $Z_{ap^s}(Z_{p^s}-tZ_{p^s-1}-tZ_{3p^s-1})$ from both sides and factoring out $tZ_{ap^s}$ on the left hand side, we obtain \begin{align*}
    tZ_{ap^s}(Z_{p^s-1}+Z_{3p^s-1})=tZ_{p^s}Z_{ap^s-1}+tZ_{(a+1)p^s}Z_{2p^s-1}.
\end{align*}

Since $Z_{bp^s-1}=Z_{(b-1)p^s+(p^s-1)}=tZ_{(b-1)p^s}Z_{p^s-1}$ for $b\le p$, the above equation becomes $$tZ_{ap^s}Z_{p^s-1}(1+tZ_{2p^s})=tZ_{p^s-1}(tZ_{p^s}Z_{(a-1)p^s}+tZ_{(a+1)p^s}Z_{p^s}).$$

Therefore, $$Z_{ap^s}=\frac{tZ_{p^s}}{1+tZ_{2p^s}}(Z_{(a-1)p^s}+Z_{(a+1)p^s}).$$

A similar computation gives $$Z_{(p-1)p^s}=\frac{tZ_{p^s}}{1+tZ_{2p^s}}Z_{(p-2)p^s}.$$
\end{proof}

For any $m\in \N$, let $M_m$ be the $m\times m$ matrix with $1$ along the main diagonal and $-t$ along the first diagonal above and below the main diagonal. Let $P_m(t)=\det M_m$, and $P_0(t)=1.$

\begin{prop} \label{Z_ap^s}
Given $s \ge 0$, for $0 \le a <p$,
  $$Z_{ap^s}=\frac{1}{t} \frac{c_s^{a}P_{p-a-1}(c_s)}{P_{p-1}(c_s)}.$$

  In particular, $$Z_{a}=\frac{1}{t} \frac{t^{a}P_{p-a-1}(t)}{P_{p-1}(t)}.$$
\end{prop}

\begin{proof}
    Given $s\ge 0$, consider a $p\times p$ matrix $M=\begin{pmatrix}
        A &B\\
        C &D
    \end{pmatrix}$ where the block $D$ is obtained by substituting $t=c_s$ in the matrix $M_{p-1}$, the block matrix $A=(1)$, the block $B$ is a zero row and the block $C$ is a column with $-c_s$ as its the first entry and $0$ for the remaining entries. As the column vector $\begin{pmatrix}
            Z_0, Z_{p^s}, \ldots,             Z_{(p-1)p^s}
        \end{pmatrix}$ is the solution of $M \vec{x} =\begin{pmatrix}
        1/t,  0, \ldots,         0
        \end{pmatrix}$ by Lemma $\ref{Z_p^s relations}$, it follows that $\begin{pmatrix}
             Z_0, Z_{p^s}, \ldots,             Z_{(p-1)p^s}
        \end{pmatrix}=M^{-1} \begin{pmatrix}
        1/t,  0, \ldots,         0
        \end{pmatrix}.$

By the definition of $M$, it follows that $\det M_{1, k}= (-c_s)^{k-1}P_{p-k}(c_s)$. Hence, the $(k, 1)$-th entry of $M^{-1}$ is $$(\det M)^{-1} (-1)^{k+1}(-c_s)^{k-1}P_{p-k}(c_s)=(P_{p-1}(c_s))^{-1}c_s^{k-1}P_{p-k}(c_s).$$ Therefore, $$Z_{ap^s}=(1/t)(c_s^{a}P_{p-a-1}(c_s)/P_{p-1}(c_s)).$$
\end{proof}

The Proposition \ref{Z_ap^s} gives the following result
\begin{lemma}\label{functionF}
    For $1 \le a <p$, define $F_a(t)=t^aP_{p-a-1}(t)/P_{p-1}(t)= tZ_a(t)$. For any non-negative integer $s$, $$tZ_{ap^s}=F_a(c_s)=c_s Z_a(c_s).$$
\end{lemma}

\subsection{The function $c_s(t)$}

Recall that $Z_n=X_{n-1}=\sum_{i=0}^\infty a_i t^i$ is the generating function associated to $T(n-1)$ so it is a function in $t$. Therefore, the function $c_s=tZ_{p^s}/(1+tZ_{2p^s}) $ is a function in $t$. Abusing notations, we write $Z_n(t)$ as $Z_n$ and $c_s(t)$ as $c_s$ when the fact that they are functions in the variable $t$ is not relevant. The notations $Z_n(c_i)$ ($c_s(c_i)$, resp.) for some $i$ denote the functions $Z_n(t)$ ($c_s(t)$, resp.) evaluated at $t=c_i$.

The next lemma relates $c_s$ to certain functions $Z_n$
with $n<p^s$ and it will be used to prove the iterative relation $c_s=c_1(c_{s-1})$.

\begin{lemma} \label{expression of c_s}
    For $s\ge 0$, the function $c_s$ is equal to $$\frac{t^2Z_{p^s-1}}{1-2t^3\sum_{i=0}^{s-1}Z_{p^i}Z_{p^i-1}}.$$ 
\end{lemma}

\begin{proof}
    Lemma \ref{Z_p^s relations} gives \begin{equation*}
    Z_{(p-1)p^s}=c_sZ_{(p-2)p^s}.
\end{equation*}

By (\ref{eq2}) and Proposition \ref{ap^s+n=ap^s+n}, 
    \begin{align} \label{eq6}
\nonumber        
        Z_{(p-1)p^s}&=t(Z_{(p-1)p^s-1}+ \sum_{i=0}^{s-1} 2Z_{(p-1)p^s+2p^i-1})\\ 
\nonumber
        &=t(Z_{(p-2)p^s+p^s-1}+ \sum_{i=0}^{s-1} 2Z_{(p-1)p^s+2p^i-1})\\
        &=t(tZ_{(p-2)p^s}Z_{p^s-1}+ \sum_{i=0}^{s-1} 2tZ_{(p-1)p^s}Z_{2p^i-1}).
    \end{align}

Substituting $ Z_{(p-1)p^s}=c_sZ_{(p-2)p^s}$ into (\ref{eq6}) and dividing it by $ Z_{(p-2)p^s}$, one gets

\begin{align*}
    c_s&=t(tZ_{p^s-1}+2tc_s \sum_{i=0}^{s-1}Z_{2p^i-1})\\
    &=t(tZ_{p^s-1}+2tc_s \sum_{i=0}^{s-1}tZ_{p^i}Z_{p^i-1})\\
    &=t^2Z_{p^s-1}+2t^3c_s\sum_{i=0}^{s-1}Z_{p^i}Z_{p^i-1}.
\end{align*}

Therefore, $(1-2t^3\sum_{i=0}^{s-1}Z_{p^i}Z_{p^i-1})c_s=t^2Z_{p^s-1}$ and \begin{equation*}
    c_s=\frac{t^2Z_{p^s-1}}{1-2t^3\sum_{i=0}^{s-1}Z_{p^i}Z_{p^i-1}}.     
\end{equation*}
\end{proof}

\begin{prop} \label{s=1+s-1}
    For $s \ge 1$, $c_s(t)=c_1(c_{s-1})$.
\end{prop}

\begin{proof}
    We prove it by induction on $s$.

    When $s=1$, it holds as $c_0=t$.

    Let $s>1$.    
    By the induction hypothesis,  $c_k(t)=c_1(c_{k-1})$ for all $k \le s-1$.

    Since $tZ_{ap^k}(t)=F_a(c_k(t))$ by Lemma \ref{functionF} and by the induction hypothesis, it follows that when $k \le s-1$, $$        c_1Z_{ap^{k-1}}(c_1)=F_a(c_{k-1}(c_1))=F_a(c_k)=tZ_{ap^k}(t).    $$ Therefore, given  $n$ an integer , $1\le n <p^{s-1}$,  Proposition \ref{ap^s+n=ap^s+n}  and the identity $c_k(t)=c_1(c_{k-1})$ imply 
\begin{equation}\label{c_1.levels.up}
    c_1Z_{n}(c_1)=tZ_{pn}(t).
\end{equation}

By Lemma \ref{expression of c_s}, $$c_s(t)=\frac{t^2Z_{p^s-1}(t)}{1-2t^3\sum_{i=0}^{s-1}Z_{p^i}(t)Z_{p^i-1}(t)}.$$
Hence, 

\begin{align*}
    c_s(c_1)&=\frac{c_1^2Z_{p^s-1}(c_1)}{1-2c_1^3\sum_{i=0}^{s-1}Z_{p^i}(c_1)Z_{p^i-1}(c_1)}\\
    &=\frac{c_1(c_1Z_{p^s-1}(c_1))}{1-2c_1\sum_{i=0}^{s-1}c_1Z_{p^i}(c_1)c_1Z_{p^i-1}(c_1)}\\
    &=\frac{c_1tZ_{p^{s+1}-p}(t)}{1-2c_1\sum_{i=0}^{s-1}tZ_{p^{i+1}}(t)tZ_{p^{i+1}-p}(t)}.
\end{align*}

Substitute $c_1=\frac{t^2Z_{p-1}}{1-2t^3Z_1Z_0}$ into the right hand side of the above equation and multiply the numerator and the  denominator by $1-2t^3Z_1 Z_0$. Then \begin{align*}
c_s(c_1)&=\frac{tZ_{p^{s+1}}t^2Z_{p-1}}{1-2t^3Z_1Z_0-2t^2Z_{p-1}\sum_{i=0}^{s-1}tZ_{p^{i+1}}tZ_{p^{i+1}-p}}\\
&\stackrel{(i)}{=}  \frac{tZ_{p^{s+1}}t^2Z_{p-1}}{1-2t^3Z_1Z_0-2t^3\sum_{i=0}^{s-1}Z_{p^{i+1}}(tZ_{p^{i+1}-p}Z_{p-1})}\\
&\stackrel{(ii)}{=}\frac{tZ_{p^{s+1}}t^2Z_{p-1}}{1-2t^3Z_1Z_0-2t^3\sum_{i=0}^{s-1}Z_{p^{i+1}}Z_{p^{i+1}-1}}\\
&=\frac{tZ_{p^{s+1}}t^2Z_{p-1}}{1-2t^3\sum_{i=0}^{s}Z_{p^{i}}Z_{p^{i}-1}},
\end{align*}
where (i) comes from moving $Z_{p-1}$ in the denominator into the summation and (ii) comes from the multiplicative property of generating functions $Z_n$.

Then the result follows from Lemma \ref{expression of c_s}. 
    \end{proof}

%\subsubsection{Chebyshev polynomial} 
For any $n\in \N$, let $Q_n(x)$ be the polynomial of degree $n$ so that $Q_n(2\cos \theta)=2\cos n \theta$, equivalently, $Q_n(t+t^{-1})=t^n+t^{-n}$. Let $T_n(x)$ be the Chebyshev polynomials of the first kind, i.e., $T_n(\cos x)= \cos nx$. Hence, $Q_n(x)=2T_n(x/2)$. When we use the term Chebyshev polynomials in this article, we use $Q_n(x)$ instead of $T_n(x).$

Recall that the prime $p>2$ is the characteristic of the algebraically closed field in the setting. Our goal is to relate $c_1(t)$ to $Q_p(t)$.

\begin{prop}
The function $c_1(t)=(Q_p(1/t))^{-1}$.
\end{prop}

\begin{proof}
    Since $tZ_1(t)=tP_{p-2}(t)/P_{p-1}(t)$ and $tZ_{p-1}(t)=t^{p-1}/P_{p-1}(t)$, by Lemma \ref{expression of c_s}, $$c_1(t)=\frac{t^2Z_{p-1}}{1-2t^2Z_1}=\frac{t(t^{p-1}/P_{p-1}(t))}{1-2t(tP_{p-2}(t)/P_{p-1}(t))}=\frac{t^p}{P_{p-1}(t)-2t^2P_{p-2}(t)}.$$

    Let $A_n(t)=P_{n-1}(t)-2t^2P_{n-2}(t)$, and $F_n(1/t)=(1/t^n)A_n(t)$. Since $c_1(t)=(F_p(1/t))^{-1}$, it suffices to show $F_p(t)=Q_p(t).$

    Let $P_0(t)$=1. By the definition of $P_n(t)$ and expansion by minors, one has $$P_n(t)=P_{n-1}(t)-t^2P_{n-2}(t)$$ when $n>1$. Since $A_n(t)$ is a linear combination of $P_{n-1}(t)$ and $P_{n-2}(t)$ (with coefficients in $\Q(t)$), $$A_n(t)=A_{n-1}(t)-t^2A_{n-2}(t).$$

    Hence, \begin{align*}
        F_n(t)&=t^nA_n(1/t)\\
        &=t^n(A_{n-1}(1/t)-t^{-2}A_{n-2}(1/t))\\
        &=t\cdot t^{n-1}A_{n-1}(1/t)-t^{n-2}A_{n-2}(1/t)\\
        &=tF_{n-1}(t)-F_{n-2}(t).
    \end{align*}

    As $Q_n(u)$ is the polynomial so that $Q_n(t+t^{-1})=t^n+t^{-n}$, we have $$Q_{n+1}(u)=uQ_n(u)-Q_{n-1}(u)$$ by letting $u=t+t^{-1}$.

Therefore, $\{F_n(t)\}_{n=2}^\infty$ and $\{Q_n(t)\}_{n=2}^\infty$ satisfy the same second order recurrence. A direct computation shows that the sequence $F_n(t)$ and $Q_n(t)$ have the same initial two values; hence, they are the same sequence.
\end{proof}

\subsection{The coefficients of $Z_n$}

    Recall that $P_m(t)=\det M_m$, where $M_m$ is the $m\times m$ matrix  with $1$ along the main diagonal, $-t$ along the first diagonal above and below the main diagonal, and $0$ elsewhere. By \cite{kulkarni1999eigenvalues}, 
    the eigenvalues of $M_m$ are $$1- 2t\cos \frac{k\pi}{m+1} \text{ where $k=1,\ 2,\ \cdots,\ m$.}$$

    Hence, for $m>0$, $$P_m(t)= \prod_{k=1}^m (1- 2t\cos \frac{k\pi}{m+1}).$$

\begin{notation}
    Define $P_0(t)=R_0(t)=1$. For $m \in \N$, let $P_m(t)= \prod_{k=1}^m (1- 2t\cos \frac{k\pi}{m+1})$ and let $R_m(t)=\prod^m_{k=1}(t-2 \cos \frac{k \pi}{m+1}).$  Let $R_{m,\ s}(t)=R_m \circ Q_{p^s}(t)$ where $Q_{p^s}(t)$ is the Chebyshev polynomial of degree $p^s$,
\end{notation}

By Proposition \ref{Z_ap^s}, for $p>a\ge0$,  $$tZ_a(t)=\frac{t^aP_{p-a-1}(t)}{P_{p-1}(t)}=\frac{t^{-(p-a-1)}P_{p-a-1}(t)}{t^{-(p-1)}P_{p-1}(t)}= \frac{R_{p-a-1}(t^{-1})}{R_{p-1}(t^{-1})}.$$

\begin{notation}
Given a monic polynomial $f(x)$ with $n$ distinct nonzero roots $\beta_1,\ \ldots,\  \beta_n$, let $\hat{f}(x)$ denote the monic polynomial with $n$ distinct roots $\beta_1^{-1},\ \ldots,\ \beta_n^{-1}.$ Hence, $\hat{\hat{f}}(x)=f(x)$ and $\hat{f}(x)=x^nf(1/x)=\prod_{i=1}^n (1-\beta_i x).$    
\end{notation}

\begin{prop}\label{Z_apinR}
    For $p>a\ge0$, \begin{equation*}
        tZ_{ap^s}(t) =\frac{t^{ap^s }\widehat{R}_{p-a-1,\ s}(t)} {\widehat{R}_{p-1,\ s}(t)}=  \frac{t^{ap^s }\prod_{\b\in A^s_{p-a-1}}(t-\b^{-1})}{\prod_{\b\in A^s_{p-1}}(t-\b^{-1})}
    \end{equation*}

    where $A_0^s= \varnothing$ and  for $0<m <p$, $$A_{m}^s =\bigsqcup_{k=1}^m \{ 2\cos \frac{2l \pi \pm (k \pi)/(m+1)}{p^s} \mid l=0, 1, \ldots, p^s-1  \}.$$
\end{prop}

\begin{proof}
    By Lemma \ref{functionF}, $tZ_{ap^s}(t)=c_sZ_a(c_s)$ where $c_s(t)=(Q_{p^s}(1/t))^{-1}.$

    Since $$tZ_a(t)= \frac{R_{p-a-1}(t^{-1})}{R_{p-1}(t^{-1})},$$

    \begin{align*}
        tZ_{ap^s}(t) =c_s Z_a(c_s)=\frac{R_{p-a-1}(c_s^{-1})}{R_{p-1}(c_s^{-1})}=\frac{R_{p-a-1}(Q_{p^s}(1/t))}{R_{p-1}(Q_{p^s}(1/t))}\\
        =\frac{t^{ap^s}\cdot t^{(p-a-1)p^s} R_{p-a-1,\ s}(1/t) }{t^{(p-1)p^s}R_{p-1,\ s}(1/t)}=\frac{t^{ap^s }\widehat{R}_{p-a-1,\ s}(t)}{\widehat{R}_{p-1,\ s}(t)}.
    \end{align*}

The set of roots of $R_m(t)$ is $\{ 2\cos \frac{k\pi}{m+1} \mid  1 \le k \le m   \}.$ Therefore, the set of roots of $R_{m, s}(t)$ is \begin{equation*}
    \{2 \cos \theta \mid Q_{p^s}(2 \cos \theta)= 2 \cos (p^s\theta)=2 \cos \frac{k \pi}{m+1},\ k \in \{1, 2, \ldots, m\}\}\end{equation*}
    and that is \begin{equation*}
        \{2\cos \theta \mid \theta= \frac{2l\pi \pm (k\pi)/(m+1)}{p^s},\ l \in \{0, 1, \ldots, p^s-1\}, \ k \in \{1, 2, \ldots, m\}   \}.
    \end{equation*}

Since $\frac{k \pi}{m+1} < \pi$, it follows that for any $1 \le k' \le m$, $$2l\pi + \frac{k \pi}{m+1} \neq 2(l+1)\pi - \frac{k' \pi}{m+1}.$$ 

Hence, the set $\{\gamma \mid \gamma= \frac{2l\pi \pm (k\pi)/(m+1)}{p^s},\ l \in \{0, 1, \ldots, p^s-1\}   \}$ has $2p^s$ elements. As $\cos x$ is an even function and decreasing on $0 < x < \pi$, the set $$\{2\cos \gamma \mid \gamma= \frac{2l\pi \pm (k\pi)/(m+1)}{p^s},\ l \in \{0, 1, \ldots, p^s-1\}   \}$$ has $p^s$ elements.

As $\hat{f}(t)$ is the monic polynomial whose roots consist of the reciprocals of the roots of $f(t)$, we are done.    
\end{proof}

Recall \cite[, Lemma 2.3]{larsen2024boundsmathrmsl2indecomposablestensorpowers}, which gives

\begin{lemma} \label{1/hat}
    Let $f(t)$ be a monic polynomial with $n$ distinct nonzero roots $\b_1, \ldots, \b_n$, and $\hat{f}(t)$ be as above. Then $$\frac{1}{\hat{f}(t)}=\sum_i \frac{\beta_i^{n-1}}{f'(\beta_i)(1-\beta_i t)}.$$
\end{lemma}

\begin{lemma} \label{g/fhat}
   Let $f(t)$ be a monic polynomial with $n$ distinct nonzero roots $\beta_1,\ \ldots,\  \beta_n$, and $g(t)=a_dt^d+\cdots+a_0$, a polynomial of degree $d.$ Then \begin{itemize}
       \item[(1)]  when $k \ge d,$  the $t^k$-coefficient of $g(t)/\hat{f}(t)$ is $$\sum_{\{ \b \mid f(\b)=0  \}}  \frac{\b^{n-1}}{f'(\b)}g(\b^{-1})\b^k.$$

       \item [(2)]  When $k<d$, the $t^k$-coefficient of $g(t)/\hat{f}(t)$ is $$\sum_{\{ \b \mid f(\b)=0  \}} \frac{\b^{n-1}}{f'(\b)}\sum_{0 \le i \le k} a_i\b^{k-i}.$$
   \end{itemize}

\end{lemma}

\begin{proof}
    \begin{itemize}
        \item [(1)]By Lemma \ref{1/hat}, $$\frac{g(t)}{\hat{f}(x)}=g(t)\sum_{\{ \b \mid f(\b)=0  \}} \frac{\beta^{n-1}}{f'(\beta)(1-\beta  t)}=\sum_{\{ \b \mid f(\b)=0  \}} \frac{\beta^{n-1}}{f'(\beta)}(a_dt^d+\cdots+a_0)(1+\b t+\b^2t^2+\cdots).$$

    Therefore, when $k \ge d$, the $t^k$-coefficient of $g(t)/\hat{f}(x)$ is $$\sum_{\{ \b \mid f(\b)=0  \}} \frac{\b^{n-1}}{f'(\b)}\sum_{0 \le i \le d} a_i\b^{k-i}=\sum_{\{ \b \mid f(\b)=0  \}} \frac{\b^{n-1}}{f'(\b)} g(\b^{-1}) \b^k.$$

        \item [(2)] follows by the same reasoning as (1) since only non-negative powers of $\b$ can appear in the sum.    
    \end{itemize}    
\end{proof}

Given a positive integer $n$, suppose $p^{s+1}>n \ge p^s$. Write $n=a_sp^{s}+a_{s-1}p^{s-1}+ \cdots + a_0 $ where $p>a_i \ge 0$ and $a_s \neq 0$. By Proposition \ref{ap^s+n=ap^s+n} and Proposition \ref{Z_apinR},  
\begin{equation*}
    tZ_n(t)=\frac{t^n \prod_{i=0}^s \widehat{R}_{p-a_i-1,i}(t)}{\prod_{i=0}^s\widehat{R}_{p-1,i}(t)}.
\end{equation*}

%Suppose $T$ is a tilting module of $G$, and $Q(x)$ is the polynomial with integer coefficients such that $Q(V)=T$. The total multiplicity of $T^{\xx k}$ is $\sum_n \m_n(Q(x)^k).$ As the functions $\m_n$ are additive, by Lemma \ref{g/fhat}, $\m_n(Q(x)^k)= \sum_j c_{n, j}Q(\b_{n, j})^k$ where $c_{n, j}$ is determined by $P_n(\b_{n, j}^{-1})$ and  $(\widehat{R_n})'(\b_{n, j})$ where $\b_{n, j}=2\cos \theta_{n, j}$ for some angle $ \theta_{n, j}$ if  the case (2) in Lemma \ref{g/fhat} is ignored.
Let $\m_n: \Z[x] \to \Z$ be the additive map sending $x^l$ to the multiplicity of $T(n)$ in $V^{\xx l}$. Therefore, $\m_{n-1}(x^l)$ is the coefficient of $t^l$ of $Z_n(t).$ We will use the following known identity \eqref{sin_identity} to give an explicit formula of $\m_{n-1}(x^l)$.

%We start with this known identity $\sin n\theta =2^{n-1} \prod_{k=0}^{n-1} \sin( \theta + k\pi/n)$ and give an explicit formula of $\m_{n-1}(x^l)$ when $l$ is large enough. After that, we show the offset in case (2) of Lemma \ref{g/fhat}, i.e. $l$ is not large enough, in our setting is negligible,

For any $\theta \in \R,$
\begin{equation}\label{sin_identity}
     \sin n\theta =2^{n-1} \prod_{k=0}^{n-1} \sin( \theta + \frac{k\pi}{n}). 
\end{equation}

\begin{comment}
\begin{proof}
    It is a known identity.

    Note that, $$\sin x= \frac{e^{ix}-e^{-ix}}{2i}=\frac{e^{-ix}}{2i}(e^{2ix}-1).$$

    Hence, \begin{align*}
        \prod_{k=0}^{n-1} \sin( \theta + k\pi/n)&=\prod_{k=0}^{n-1}\frac{e^{-i(\theta+k\pi/n)}}{2i}(e^{2i(\theta+k\pi/n)}-1)\\
        &=\prod_{k=0}^{n-1}\frac{e^{-i(k\pi/n)}}{2i}e^{i\theta}(e^{2ik\pi /n}-e^{-2i\theta})\\
        &=(1/2i)^n (\prod_{k=0}^{n-1}e^{-i(k\pi/n)})(e^{in\theta})(\prod_{k=0}^{n-1}e^{2ik\pi/n}-e^{-2i\theta})\\
        &=(1/2i)^n (e^{-in\pi/2}e^{i\pi/2})e^{in\theta}(-1)^n(e^{-i2n\theta}-1)\\
        &=(1/2)^{n-1} \sin (n\theta).
    \end{align*}        
\end{proof}
\end{comment}

\begin{prop} \label{R_m-1}
    For any $\theta \in \R,$ $$R_{m-1}(2\cos \theta)=\frac{2\sin m (\theta-\pi)}{\sin \theta}=\begin{dcases}
        \frac{2\sin m \theta}{\sin \theta} \qquad \text{ if $m$ is even}\\
        \frac{-2\sin m \theta}{\sin \theta} \qquad \text{ if $m$ is odd}.
    \end{dcases}$$
\end{prop}

\begin{proof}
    Define $\widetilde{R}_{m-1}(t)=\prod_{k=0}^{m-1}(t-2\cos k\pi/m)=R_{m-1}(t)\cdot(t-2).$  By the sum-to-product identity, \begin{align*}
        (\cos \theta- \cos k\pi/m)&=-2\sin(\frac{\theta}{2}+\frac{k\pi}{2m})\sin (\frac{\theta}{2}-\frac{k\pi}{2m})\\
        &=-2 \sin (\frac{\theta - \pi}{2}+\frac{(m+k)\pi}{2m})\sin (\frac{\theta - \pi}{2}+\frac{(m-k)\pi}{2m}).
    \end{align*}

    Hence, \begin{align*}
    \widetilde{R}_{m-1}(2\cos \theta)&=(-4)^m \prod_{k=0}^{m-1}\sin(\frac{\theta - \pi}{2}+\frac{(m+k)\pi}{2m})\sin(\frac{\theta - \pi}{2}+\frac{(m-k)\pi}{2m})\\
    &=(-4)^m \Big(\prod_{k=0}^{2m-1}\sin(\frac{\theta - \pi}{2}+\frac{k\pi}{2m})\Big)\sin(\frac{\theta - \pi}{2}+\frac{\pi}{2}) \big/ \sin(\frac{\theta - \pi}{2}+\frac{0\pi}{2m})\\
    &=(-4)^m \big(\prod_{k=0}^{2m-1}\sin(\frac{\theta - \pi}{2}+\frac{k\pi}{2m})\Big) (-\tan(\frac{\theta}{2})).
    \end{align*}

By \eqref{sin_identity}, 
\begin{align*}
\widetilde{R}_{m-1}(2\cos \theta)&= (-4)^{m} 2^{-(2m-1)} \sin (2m(\theta-\pi)/2) (-\tan(\theta/2))\\
&=-2 \sin (m(\theta-\pi)) \tan(\theta/2)
\end{align*}

As $R_{m-1}(2\cos \theta)=\widetilde{R}_{m-1}(2\cos \theta)(2\cos \theta-2\cos 0)^{-1}$ and $(2\cos \theta-2\cos 0)=-2\sin^2(\theta/2)$, the result follows.
\end{proof}

\begin{lemma} \label{R'p-1}
    Suppose $2\cos \theta$ is a root of $R_{p-1}(t)$, i.e. $\theta=k\pi/p$ for some $0<k<p$. Then $$R_{p-1}'(2\cos \theta)=(-1)^k \frac{p}{2\sin^2 \theta}.$$
\end{lemma}

\begin{proof}

    Let $\zeta=\zeta_{2p}=e^{i \pi/p}.$ Hence, $2\cos \theta= \zeta^k+\zeta^{-k}.$

    \begin{align*}
        R_{p-1}'(2\cos \theta)&= \prod_{l \neq k, \text{ and } 0<l<p}(\zeta^k + \zeta^{-k}-\zeta^l - \zeta^{-l})\\
        &=\prod_{l \neq k}\zeta^{-k}(1-\zeta^{k+l})(1-\zeta^{k-l})\\
        &=\zeta^{-k(p-2)}\Big(\prod_{l=1}^{2p-1}(1-\zeta^l)\Big)(1-\zeta^{2k})^{-1}(1-\zeta^k)^{-1}(1-\zeta^{p+k})^{-1}\\
        &=\zeta^{-k(p-2)}\Big(\prod_{l=1}^{2p-1}(1-\zeta^l)\Big)(1-\zeta^{2k})^{-2}\\
        &=\zeta^{-k(p-2)}\Big(\prod_{l=1}^{2p-1}(1-\zeta^l)\Big)\zeta^{-2k}(\zeta^{-k}-\zeta^k)^{-2}\\
        &=\zeta^{-kp}\frac{2p}{(2\sin \theta)^2}
        =(-1)^k \frac{p}{2 \sin ^2 \theta}.
    \end{align*}
\end{proof}

\begin{lemma} \label{R_p-1,s'}
    Suppose $\theta_0$ is an angle such that  $R_{p-1,i}(2\cos \theta_0)=0$. Hence, $p^i \theta_0=k\pi/p$ for some $k$ where $p>k>0.$  Then $$R_{p-1,i}'(2\cos \theta_0)=(-1)^k \frac{p^{i+1}}{2\sin (p^i\theta_0) \sin \theta_0}.$$
\end{lemma}

\begin{proof}
    Since $R_{p-1,i}(t)=R_{p-1}(Q_{p^i}(t)),$ it follows $$R_{p-1,i}'(2\cos \theta)=R_{p-1}'(Q_{p^i}(2\cos \theta))\cdot Q_{p^i}'(2 \cos \theta).$$

    Let $F(\theta)= Q_{p^i}(2\cos \theta)=2\cos (p^i \theta)$. By the chain rule,  $$F'(\theta)=Q_{p^i}'(2\cos \theta) \cdot (-2\sin \theta)=-2 \sin (p^i \theta)\cdot p^i.$$

    Hence, $$Q_{p^i}'(2\cos \theta)= \frac{p^i \sin(p^i \theta)}{\sin \theta},$$ and the result follows from Lemma \ref{R'p-1}.
\end{proof}

Let $n=a_sp^{s}+a_{s-1}p^{s-1}+ \cdots + a_0$. Then $$Z_n(t)=\frac{t^{n-1}\prod_{i=0}^s\widehat{R}_{p-a_i-1,\ i}(t)}{\prod_{i=0}^s\widehat{R}_{p-1,\ i}(t)}.$$ 
Let $A_n(t)=\prod_{i=0}^sR_{p-a_i-1,\ i}(t)$ and $B_n(t)= \prod_{i=0}^s R_{p-1,\ i}(t),$ of degree $\sum_{i=0}^s (p-a_i-1)p^{i}=p^{s+1}-1-n$ and $\sum_{i=0}^s (p-1)p^{i}=p^{s+1}-1$ respectively.

\begin{prop} \label{m_ninbeta}

        Suppose $l \ge n-1+ \deg A_n(t)=
        p^{s+1}-2.$ Then $$\m_{n-1}(x^l)=\sum_{\{\b \mid B_n(\b)=0\}}\frac{A_n(\b)}{B_n'(\b)}\b^l.$$ In fact, 
        \begin{equation*}
        \m_{n-1}(x^l)=\begin{dcases}
            2 \sum_{\{\b \mid B_n(\b)=0 \text{ and $\b>0$}\}}\frac{A_n(\b)}{B_n'(\b)}\b^l \quad \text{if $n-1  \equiv l$ (mod $2$)}\\
            0\qquad \qquad \qquad \qquad \quad \ \   \text{if $n-1 \not \equiv l$ (mod $2$)}.
        \end{dcases}
    \end{equation*}
 
\end{prop}

\begin{proof}
    By Lemma \ref{g/fhat}, \begin{align*}
        \m_{n-1}(x^l)&=\sum_{\{\b \mid B_n(\b)=0\}} \frac{\b^{p^{s+1}-2}}{B_n'(\b)}\b^{1-n} \prod_{i=0}^s\widehat{R}_{p-a_i-1,\ i}(\b^{-1})\b^l\\
        &=\sum_{\{\b \mid B_n(\b)=0\}} \frac{\b^{p^{s+1}-1-n}}{B_n'(\b)} \prod_{i=0}^s\widehat{R}_{p-a_i-1,\ i}(\b^{-1})\b^l.        
    \end{align*}

    As $$f(t)=\hat{\hat{f}}(t)=t^{\deg \hat{f}}\hat{f}(1/t)=t^{\deg f}\hat{f}(1/t)$$ and $$\deg A_n=p^{s+1}-1-n,$$ $$\m_{n-1}(x^l)=\sum_{\{\b \mid B_n(\b)=0\}}\frac{A_n(\b)}{B_n'(\b)}\b^l.$$

    Rewrite $\m_{n-1}(x^l)$ as
    \begin{align*}
                \sum_{\{\b \mid B_n(\b)=0 \text{ and } \b>0\}}\frac{A_n(\b)}{B_n'(\b)}\b^l+\sum_{\{\b \mid B_n(\b)=0 \text{ and } \b>0\}}\frac{A_n(-\b)}{B_n'(-\b)}(-\b)^l
    \end{align*}

    Recall $R_{m-1}(2\cos \theta)= \prod_{k=1}^{m-1}(2 \cos \theta - k\pi/m)=2\sin (m (\theta-\pi))/ \sin \theta$.

    As $\b=2 \cos \theta$,  $-\b=-2 \cos \theta=2 \cos (\pi-\theta)$. Since $\sin (\pi-\theta)=\sin \theta$, \begin{align*}
        R_{m-1}(-2\cos \theta)&=R_{m-1}(2 \cos (\pi-\theta))=\frac{2\sin (m ((\pi-\theta)-\pi))}{\sin (\pi-\theta)}\\
        &=\frac{-2 \sin (m\theta)}{ \sin (\pi-\theta)}=\begin{cases}
            R_{m-1}(2 \cos \theta), \qquad \text{when $m$ is odd}\\
            -R_{m-1}(2\cos \theta), \qquad \text{when $m$ is even.}
        \end{cases}
    \end{align*}

    Therefore, $R_{p-a-1}(2 \cos (\pi - \theta))=R_{p-a-1}(2 \cos \theta)$ if and only if $a$ is even.

    As $A_n(t)=\prod_{i=0}^sR_{p-a_i-1,\ i}(t)$, $A_n(2 \cos (\pi-\theta))= (-1)^{\sum_{i=0}^s a_i } A_n (2 \cos \theta)$.

    By Lemma \ref{R_p-1,s'}, $$R_{p-1, i}'(2\cos \theta)= (-1)^k p^{i+1}/2\sin (p^i\theta) \sin \theta$$ where $k=p^{i+1}\theta\pi$, i.e. $p^i \theta=k\pi/p$. 
    
    As $p^i(\pi-\theta) \equiv \pi-k\pi/p$ (mod $2\pi$), $$R_{p-1, i}'(2 \cos (\pi-\theta))=(-1)^{k+1}p^{i+1}/2\sin (p^i\theta) \sin \theta$$ where the exponent of $-1$ is $k+1$ since $p^i(\pi-\theta) \equiv \pi-k\pi/p=(p-k)\pi/p$ (mod $2\pi$), and $p-k\equiv k+1$ (mod $2$). Therefore, $B_n'(2\cos (\pi-\theta))=-2B_n'(2 \cos \theta).$

    Combine the results for $A_n(t)$, $B_n'(t)$ and $n \equiv \sum_{i=0}^s a_i$ (mod $2$) to obtain

    \begin{equation*}
        \frac{A_n(-\b)}{B_n'(-\b)}=\begin{dcases}
            \frac{A_n(\b)}{B_n'(\b)}  \qquad  \text{if $n$ is odd}\\
            -\frac{A_n(\b)}{B_n'(\b)} \quad \ \text{if $n$ is even}.
        \end{dcases}                  
    \end{equation*}

    Therefore

    \begin{equation*}
        \frac{A_n(-\b)}{B_n'(-\b)}(-\b)^l=\begin{dcases}
            -\frac{A_n(\b)}{B_n'(\b)}\b^l  \qquad  \text{if $n\equiv l$ (mod $2$)}\\
            \frac{A_n(\b)}{B_n'(\b)}\b^l \quad \ \text{if $n\not \equiv l$ (mod $2$)},
        \end{dcases}                  
    \end{equation*}
    and the result follows.
\end{proof}

%Let $n$, $a_i$, and $i$ be defined as before. Let $0<m\le p$. Let $s=s_1$ and $\theta_{s, j}=j\pi/p^{s+1}.$ In the following, we give estimations of $R_{m-1}(2 \cos \theta)=\sin (m\theta)/\sin \theta$ and $\prod_{i=0}^s R_{n_i-1,i}(2 \cos \theta)$ for any $\theta$ of the form $\theta_{s, j}$. After that, we estimate $(\prod_{i=0}^sR_{p_i-1,i})'(2 \cos \theta)$ where $2 \cos \theta$ is a root of $\prod_{i=0}^sR_{p_i-1,i}(t).$

\begin{lemma} \label{est_R_m-1} Let $0<m\le p$ and $s\in \N$. Let $\theta=\theta_{s, j}=j\pi/p^{s+1}$ for some integer $j$. Then
    \begin{itemize}
        \item [(1)] $\vert \sin \theta \vert \ge \vert \sin \theta_{s,1}\vert \ge 2/p^{s+1}.$
        \item[(2)] $2^p>\vert R_{m-1}(2 \cos \theta) \vert \ge 4/p^{s+1}.$
        \item[(3)] Suppose $\pi/2 > m\theta >0$. Then $2^p\ge\vert R_{m-1}(2 \cos \theta) \vert \ge 2/\pi.$
    \end{itemize}
   % 1) $2^p>\vert R_{m-1}(2 \cos \theta) \vert \ge 4/p^{s+1}.$\\
   % 2) Suppose $\pi/2 > m\theta >0$. Then $2^p\ge\vert R_{m-1}(2 \cos \theta) \vert \ge 2/\pi.$
\end{lemma}

\begin{proof}
Part (1) It follows from the fact that $\theta$ is a multiple of $\pi/ p^{s+1}$ and $\sin x > 2x/\pi $ when $\pi/2>x> 0$.

 (2) As $\sin (m\theta) =2^{m-1} \prod_{k=0}^{m-1} \sin( \theta + k\pi/m)$, one has $\vert \sin (m \theta)/\sin \theta \vert \le 2^{m-1}.$   

From (1), we obtain $\vert \sin (m \theta) \vert > 2/p^{s+1}.$ Therefore,  $$2^p \ge 2^m > \vert R_{m-1}(2 \cos \theta) \vert = \bigg\vert \frac{2\sin(m \theta)}{\sin \theta} \bigg\vert     \ge \vert 2 \sin (m \theta) \vert > 4/p^{s+1}.$$  
\\
(3) Suppose $\pi/2 > m\theta >0$. Then \begin{align*}
    2^p > \vert R_{m-1}(2 \cos \theta) \vert \stackrel{(*)}{>}  \frac{2m\theta/\pi}{\theta}=2m/\pi \ge 2/\pi,  
\end{align*}

where $(*)$ comes from Proposition \ref{R_m-1} and $x \ge \sin x$ for all $x \ge 0$.
\end{proof}

\begin{lemma}   
 \label{estPn}
    Given $s \in \Z_{\ge 0}$ and $(n_1, n_2, \cdots, n_s) \in \N^s$, $p \ge n_i \ge 1$.

    Then $$(2^p)^s \ge  \vert\prod_{i=0}^s R_{n_i-1, i}(2 \cos \theta_{s, j}) \vert \ge \pi^{-(s-1-\log_p j)}2^s p^{-(\log_p j)^2}.  $$
\end{lemma}

\begin{proof}
    Since $\vert R_{m-1}(2 \cos \theta) \vert$ is always bounded above by $2^p$, we get the upper bound.

    The lower bound of $\vert R_{m-1}(2 \cos \theta) \vert$ is separated into the cases $m \theta > \pi/2$ or not.  As $R_{n_i-1, i}=R_{n_i-1}\circ Q_{p^{i}}$, we divide into cases depending on whether $p^{i}n_i \theta_{s,j}> \pi/2$ or not.

    Suppose $p^im \theta_{s, j}>\pi/2$ for some $i<s$ and $p \ge m >0$. Therefore, $$p^imj/p^{s+1}>1/2.$$

    This gives $i+ \log _p mj>s+1-\log_p2.$ Therefore, $i>s+1 - \log_p (2mj).$ There are at most $\lfloor \log_p 2mj \rfloor$ options for such $i$. 

    As $p\ge m \ge 1$,  there are at most $\lfloor \log_p j \rfloor +2$ options: $s-1,\ s-2,\ \cdots,\ s-2-\lfloor \log_p j \rfloor$. Since $\vert \sin (p^{i} m \theta_{s, j}) \vert \ge 2 \frac{p^{i}}{p^{s+1}} $, we have the lower bounds 
    $$2p^{-2},\ 2p^{-3},\ \cdots,\ 2p^{-(3+ \lfloor \log_p j \rfloor )}.$$

    Combining the lower bounds above and Lemma \ref{est_R_m-1} (3), we have $$
    (2^p)^s \ge \vert \prod_{i=0}^s R_{n_i-1, i}(2\cos \theta_{s, j}) \vert \ge \pi^{-(s-\log_pj -1)}2^s p^{-(\log_p j)^2}.    
    $$
    
\end{proof}

\begin{lemma} \label{estR'}
    Given $s \in \Z_{\ge 0},$ let $B(t)=\prod_{i=0}^s R_{p-1, i}(t)$. Suppose $\theta_{s,j}=k\pi/p^{i_0+1}$ for some $i_0 \in \{1, 2, \cdots, s\}$ and $p>k>0$, an angle such that $2\cos \theta_{s,j}$ is a root of $B(t).$

    Then \begin{align*}
        \frac{p^{i_0+1}}{2\sin(k\pi/p)\sin \theta_{s, j}}(2^p)^{s-1} &\ge \vert B'(2 \cos \theta_{s, j}) \vert \\ &\ge  \frac{p^{i_0+1}}{2\sin(k\pi/p)\sin \theta_{s, j}}\pi^{-(s-2-\log_pj )}2^{s-1} p^{-(\log_p j)^2}.
    \end{align*}
\end{lemma}

\begin{proof}
    This follows from Lemma \ref{R_p-1,s'} and Lemma \ref{estPn}.
\end{proof}

\begin{prop} \label{m_ninbest}
    Suppose $l < \sum_{i=0}^s (p-1)p^{i}-1=p^{s+1}-2$. Then $$\m_{n-1}(x^l)=\bigg(\sum_{\{\b \mid B_n(\b)=0\}}\frac{A_n(\b)}{B_n'(\b)}+f_l(\b^{-1}) \bigg)\b^l,$$
        for some $f_l(\b^{-1}) \in o(k^{-2-\alpha}(1+\epsilon)^k),$        where $k$ is some large number so that $n < \sqrt{k} \log k$ and  $\alpha$ and $\epsilon$ are any positive number.   
\end{prop}

\begin{proof}
    By Lemma \ref{g/fhat} (2), $$\m_{n-1}(x^l)=\sum_{\{\b \mid B_n(\b)=0\}}\frac{A_n(\b)}{B_n'(\b)}\b^l -\sum_{\{\b \mid B_n(\b)=0\}} \frac{\b^{p^{s+1}-1-n}}{B_n'(\b)}\sum_{{l<i \le d}} c_i \b^{l-i},$$ where $d=p^{s+1}-2$ and $c_i$ is the $t^i$-coefficient of $t^{n-1}\widehat{A}_n(t)$.  We want to show that the latter term is small compared to $k^{-2-\alpha_p}(1+\epsilon)^k$ when $k$ is large. 
    
    The set of  coefficients of $t^{n-1}\widehat{A}_n(t)$ is the set of  coefficients of $\widehat{A}_n(t)$. By the formula $\hat{f}(t)=t^{\deg f}f(1/t)$, the set of the coefficients of $\widehat{A}_n(t)$ is the set of the coefficients of $A_n(t)$. Since $A_n(t)$ is a polynomial of degree $\sum_{i=0}^s (p-a_i-1)p^{i}<pn=\sum_{i=0}^s a_i p^{i+1}$ with roots of the form $2 \cos \theta$ for some $\theta$, $$\vert c_i \vert < \binom{pn}{pn/2} 2^{pn}.$$

    By Stirling approximation, there exists some $C>0$ so that for each $i$, $$\vert c_i \vert < C \frac{2^{pn}}{\sqrt{pn\pi/2}}2^{pn}.$$
    
    As $\vert \b \vert =\vert 2 \cos j\pi/p^{s+1}\vert > 4j/p^{s+1},$ one has $\vert \b ^{-1} \vert < p^{s+1}<pn.$ Hence, $$\sum_{{l<i \le d}} \vert c_i \b^{l-i} \vert <  pn (C \frac{4^{pn}}{\sqrt{pn\pi/2}} pn^{pn})=C \frac{4^{pn}}{\sqrt{pn\pi/2}} pn^{pn+1}<C4^{pn}pn^{pn+1}.$$

    Suppose $n< \sqrt{k}\log k $ for some large $k$. We want to show that for any $\alpha>0$ and $\epsilon>0$  \begin{equation} \label{todiscard}
        \sum_{\{\b \mid B_n(\b)=0\}} \frac{\b^{p^{s+1}-1-n}}{B_n'(\b)}\sum_{{l<i \le d}} c_i \b^{-i} \in o(k^{-2-\alpha}(1+\epsilon)^k).
    \end{equation}

    The bound for each term:

    \begin{itemize}
        \item $\vert B_n'(\b) \vert^{-1}:$ By Lemma \ref{estR'}, \begin{align*}
            \vert B_n'(\b) \vert^{-1} \le \bigg(\frac{p^{i_0+1}}{2\sin(k\pi/p)\sin \theta_{s, j}}\pi^{-(s-2-\log_pj )}2^{s-1} p^{-(\log_p j)^2} \bigg)^{-1}\\
            \le 2\pi^{s-\log_p j}2^{1-s}p^{(\log_p j)^2}
        \end{align*}
        for some $j<p^{s+1}<pn.$ Hence $$\vert B_n'(\b) \vert^{-1} \le \pi^s (pn)^{pn} \le \pi^n(pn)^{pn}.$$

        \item $\b^{{p^{s+1}-1-n}} < 2^{pn}$
        
        \item $\sum_{{l<i \le d}} \vert c_i \b^{l-i} \vert <  C \frac{4^{pn}}{\sqrt{2pn\pi/2}} pn^{pn+1}.$
    \end{itemize} 

    Let $N=\sqrt{k}$. By assumption $n < \sqrt{k}\log k=2N \log N$, we rewrite the bounds in $N$,

    \begin{itemize}
        \item $\vert B_n'(\b) \vert^{-1}< \pi ^{2N \log N}(2pN \log N)^{2pN\log N}$

        \item $\b^{-n+\sum_{i=0}^s(p-1)p^{i}}< 2^{2pN \log N} $

        \item  $\sum_{{l<i \le d}} \vert c_i \b^{l-i} \vert < C4^{2pN\log N}(2pN \log N)^{2p N \log N}$
    \end{itemize}

To show \eqref{todiscard}, it suffices to show 

$$\lim_{k \to \infty} \bigg(\sum_{\{\b \mid B_n(\b)=0\}} \frac{\b^{p^{s+1}-1-n}}{B_n'(\b)}\sum_{{l<i \le d}} c_i \b^{-i}\bigg) (k^{-2-\alpha}(1+\epsilon)^k)^{-1}=0$$

Using the bounds in terms of $N$ and taking the logarithm, it shows that the limit is zero.

%$$ \lim_{N \to \infty}  \frac{[\pi ^{2N \log N}(2pN \log N)^{2pN\log N \cdot 2^{2pN \log N} \cdot (2pN \log N)2^{4p N \log N}(2pN \log N)^{2p N \log N}}]N^{4+2 \alpha}}{(1+ \epsilon)^{N^2}}=0.$$

%It follows from considering the difference of the logarithm of the numerator and the logarithm of the denominator when $N$ approaches to infinity.
\end{proof}

\section{Main result}

Let $K$ be an algebraically closed field of $\ch$ $p>0$. Denote $\SL_2(K)$ by $G$. Let $V$ denote the natural representation of $G$. As $\text{Tilt($G$)}$,  the ring of virtual tilting modules of $G$, is a subring of the representation ring of $G$, there is a ring homomorphism from $\Z[x] \to \text{Tilt($G$)}$ defined by $Q(x) \mapsto Q(V)$. 

As $t^n+t^{-n}=Q_n(t+t^{-1})$ where $Q_n(x)\in \Z[x]$ is the Chebyshev polynomial of degree $n$, the map $Q(x) \mapsto Q(t+t^{-1})$ gives a ring  isomorphism from $\Z[x]$ to $\Z[t, t^{-1}]^{\Z/2\Z}$, the invariant subring  of $\Z[t, t^{-1}]$ where the non-trivial element of $\Z/2\Z$ sends $t$ to $t^{-1}$. 

Given any tilting module $T$, since $\chi_T(t)$, the formal character of $T$, is an element of $\Z[t, t^{-1}]^{\Z/2\Z}$, it follows that $\chi_T(t)=Q(t+t^{-1})$ for some unique $Q(x) \in \Z[x]$. As the formal character of $Q(V)$ is $Q(t+t^{-1})=\chi_V(t)$ and the tilting modules are determined by their formal characters, $Q(V)=T$. Hence, the ring $\Z[x]$ is isomorphic to the ring $\text{Tilt($G$)}.$

We modify Lemma 6.2 from   \cite{larsen2024boundsmathrmsl2indecomposablestensorpowers} which deals with the case $p=2$ to the case $p>2$.

\begin{lemma} \label{lemma_p}
    Suppose $K$ is of characteristic $p>2.$ Let $T$ be a tilting representation of $G$ and $Q(x)\in \Z[x]$ such that $Q(V)=T.$
    Then
    \begin{itemize}
        \item [(1)] $Q(2)=\dim T$
        \item [(2)] $Q'(2)>0$
        \item[(3)] $\vert Q(x) \vert < \dim T$ for all $x \in (-2, 2)$
        \item[(4)] $\vert Q(-2) \vert =\dim T$ if and only if $T$ is purely even or purely odd, i.e., a direct sum of tilting representations whose highest weights are all even or all odd. More precisely, if $Q(-2)=\dim T$ ($Q(-2)=-\dim T$, resp.), then $T= \oplus_mT(m)$ for some even (odd, resp.) $m$. 
    \end{itemize}
\end{lemma}
\begin{proof}
    For (4), we note that by the character formula $\eqref{char_n}$, the form of elements in $\supp (n)$, and $2p^s \equiv 0$ (mod $2$), it follows that when we write $\chi_n(t)=\sum_m a_mt^m$ where $a_m \neq 0$, all the $m$ are even if $n$ is even, and all $m$ are odd if $n$ is odd. With this observation, (1)--(4) follow as in \cite{larsen2024boundsmathrmsl2indecomposablestensorpowers}.  
    %For (4), as $Q(t+t^{-1})$ is the formal character of the tilting module $T$, the coefficients of $Q(t+t^{-1})$ are non-negative. Therefore, $\vert Q(-2) \vert = Q(2) = \dim T$ if and only if all the parity of $m$ for those $t^m$ with a nonzero coefficient are the same. 
    
    %Let $\chi_n(t)$ be the character of $T(n)$. Then $Q(t+t^{-1})=\sum_n \chi_n(t)$ for some $n$. By the character formula $\eqref{char_n}$, the form of elements in $\supp (n)$, and $2p^s \equiv 0$ (mod $2$), it follows that when we write $\chi_n(t)=\sum_m a_mt^m$ where $a_m \neq 0$, all the $m$ are even if $n$ is even, and all $m$ are odd if $n$ is odd. 
\end{proof}

\begin{thm}
    Let $K$, $T$ and $Q(x)$ be as in the above lemma. Then for each $T$, there exists some positive numbers $C_T$ and $D_T$ such that when $k$ is sufficiently large, 

    \begin{equation} \label{main_ineq}
        C_Tk^{-\alpha_p} (\dim T)^k < b_k^{G, T}< D_T k^{-\alpha_p}(\dim T)^k    
    \end{equation}
    where $\alpha_p=1-(1/2)\log_p(\frac{p+1}{2})$.    
    
\end{thm}

\begin{proof}
    We may assume $\deg Q(x)>0$, i.e, $T$ is not a direct sum of the trivial representations; otherwise, $b_k^{G, T}=(\dim T)^k.$  We treat the case that $\vert Q(-2) \vert \neq \dim T$ first. Then we prove the case $\vert Q(-2) \vert = \dim T$.

    The formal character of $T$ is $Q(t+t^{-1})$ and the formal character of $T^{\xx k}$ is $Q(t+t^{-1})^k$. We will use Hoeffding's inequality to show that the space
    \begin{equation*}
        W=\langle v \mid v \text{ is a weight vector with a weight higher than $\sqrt{k}\log k$ } \rangle_K
    \end{equation*}
    is of dimension smaller than $(\dim T)^k\exp( -C (\log k)^2)$ for some  $C>0$. Since the weight $n$ subspace of $T(n)$ is of dimension $1$, it follows that $$\sum_{n> \sqrt{k}\log k } \m_n (Q(x)^k)<(\dim T)^k\exp( -C (\log k)^2).$$ Therefore, to prove the theorem, it suffices to show that $$\sum_{n< \sqrt{k}\log k } \m_n (Q(x)^k)$$ have an upper bound and a lower bound of the form in \eqref{main_ineq}.
    
    Write the formal character of $T$ as $\sum_{-d \le i \le d} a_i t^i$.  Let $X_i,\ \cdots,\ X_k$ be i.d.d. random variables, with outcomes $-d\le j\le d$ with probability $a_j/\dim T$ . Let $S_k=\sum_{i=1}^k X_i$ and $E_k$ be the expected value. The expected value $E_k=0$ by its definition. By Hoeffding's inequality, \begin{align*}
        \text{Pr } \big[  S_n > \sqrt{k}\log k \big] &\le  \exp (-2k(\log k)^2/4d^2k) \\
        &<\exp(-C(\log k)^2)
    \end{align*}
 for some $C>0.$
    Hence $$\dim W < (\dim T)^k\exp( -C (\log k)^2).$$

    As the map $\m_n$ is additive and by Lemma \ref{g/fhat}, $\m_{n-1}(Q(x)^k)=\sum_j a_{n, j}Q(\b_{n, j})^k$ for some $a_{n,j}.$ 

    For any $\alpha>0$, by Proposition \ref{m_ninbeta} and Proposition \ref{m_ninbest},  we have
     \begin{align*}
        \sum_{n< k \log k} \m_{n-1}(Q(x)^k)&=\sum_{n< k \log k} \sum_{\{j \mid B_n(\b_{n, j})=0\}} a_{n, j}Q(\b_{n, j})^k\\
        &=\sum_{n< k \log k} \sum_{\{j \mid B_n(\b_{n, j})=0\}}\Big( c_{n, j}Q(\b_{n, j})^k  +o(k^{-2-\alpha} (\dim T)^k )\Big)
    \end{align*}

    where $c_{n, j}=A_n(\b_{n,j})/B_n'(\b_{n,j})$, $A_n(t)$, $B_n(t)$ are as in Proposition \ref{m_ninbeta}, and $\b_{n, j}=2\cos \theta_{n, j}$.

    By Lemma \ref{lemma_p}, $Q'(2)>0$ and $\vert Q(x) \vert < \dim T$ for $x \in (-2, 2)$. 
    
    Let $\delta >0$. Since $\vert Q(-2) \vert \neq \dim T,$  $\max _{\{x\in [-2, 2-\delta]\}} \vert Q(x) \vert < \dim T$. Therefore,  for any $\b\in [-2, 2-\delta]$, $\vert Q(\b) \vert < (1-\epsilon)\dim T$ for some $\epsilon>0.$ Hence, for those angles $\theta_{n, j}$ not small enough, $\vert Q(\b_{n, j})^k \vert < (1-\epsilon)^k(\dim T)^k.$

    A calculation similar to Proposition \ref{m_ninbest} gives that $c_{n, j}\in o(k^{-2-\alpha}(1+\delta')^k)$ for any $\delta'>0$ Hence, $c_{n, j}Q(b_{n, j})^k \in o(k^{-2-\alpha} (\dim T)^k )$ for any $\b_{n, j}\in [-2, 2-\delta].$

    Now we take $\delta=\log k/k>0$ and we may assume $\b_{n, j}\in (2-\delta, 2).$

    By Taylor's Theorem for $Q(x)$ at $x=2$ and $\log(1+x)$ at $x=0$,   \begin{equation}
        Q(\b_{n, j})=\dim T - Q'(2) (2- \b_{n, j})+ O(\delta^2) 
    \end{equation}
    so
    \begin{equation} \label{choosel}
        k(\log Q(\b_{n,j})- \log \dim T )=l(\log \b_{n, j}- \log 2)+O((\log^2 k)/k)
    \end{equation}
where $l=\lfloor 2kQ'(2)/ \dim T \rfloor$ or $\lfloor 2kQ'(2)/ \dim T \rfloor+1$ depending on the desired parity of $l$.

Therefore, \begin{equation} \label{M_k}
    \bigg(\frac{Q(\b_{n, j})}{\dim T} \bigg)^k=\big(\frac{\b_{n, j}}{2} \big)^l \cdot M_k
\end{equation}
    for some $M_k$ where $\lim_{k \to \infty} M_k=1.$

    As $\m_{n-1}(Q(x)^k)= \sum_j \Big(c_{n, j} Q(\b_{n, j})^k+o(k^{-2-\alpha}(\dim T)^k)\Big)$, we have 
    \begin{align*}
        \m_{n-1}(Q(x)^k) &= \bigg (\sum_{\{ j \mid \b_{n,j} \in (2-\delta, 2)  \}}  c_{n,j} Q(\b_{n, j})^k \bigg) +o(k^{-1-\alpha}(\dim T)^k)\\
        &=\bigg( (\dim T)^k (2^{-l})M_k \sum_{\{ j \mid \b_{n,j} \in (2-\delta, 2)  \}} c_{n, j}(\b_{n, j})^l \bigg) +o(k^{-1-\alpha}(\dim T)^k).        
    \end{align*}

For each $n$, we choose $l$ in \eqref{choosel} so that $l$ and $n-1$ have the same parity. Since $l \equiv n-1$ (mod $2$),  by Proposition \ref{m_ninbeta}, $$\m_{n-1}(Q(x)^k)= 1/2(\dim T)^k (2^{-l})M_k \cdot \m_{n-1}(x^l) +o(k^{-1-\alpha}(\dim T)^k).$$

Then \begin{align*}
    b_k^{G, T}&=\sum_n \m_{n-1}(Q(x)^k)\\
    &=\sum_{\text{$n$ is odd}} \m_{n-1}(Q(x)^k)  +\sum_{\text{$n$ is even}} \m_{n-1}(Q(x)^k) \\
     &= \bigg( 1/2(\dim T)^k (2^{-l})M_k \sum_{\text{$n$ is odd}} \m_{n-1}(x^l) +  1/2(\dim T)^k (2^{-(l+1)})M_k \sum_{\text{$n$ is even}} \m_{n-1}(x^{l+1}) \bigg) +o(k^{-\alpha}(\dim T)^k)\\
     &=  \bigg( 1/2(\dim T)^k (2^{-l})M_k \sum_{n \in \N} \m_{n-1}(x^l) +  1/2(\dim T)^k (2^{-(l+1)})M_k \sum_{n \in \N} \m_{n-1}(x^{l+1}) \bigg) +o(k^{-\alpha}(\dim T)^k).   
\end{align*}

By \cite[, Main Theorem 1B.8]{coulembier2024fractalbehaviortensorpowers},

\begin{equation*}
    Cl^{-\alpha_p}2^l < \sum_{ n< \sqrt{l}\log l } \m_n(x^l) < D l^{-\alpha_p} 2^l
\end{equation*}
where $\alpha_p=1-(\log_p \frac{p+1}{2})/2.$

Therefore,

\begin{equation*}
    (1/2)CM_k (l^{-\alpha_p}+(l+1)^{-\alpha_p}) (\dim T)^k <  \sum_{ n< \sqrt{k}\log k } \m_n(Q(x)^k)<(1/2)DM_k (l^{-\alpha_p}+(l+1)^{-\alpha_p})(\dim T)^k.
\end{equation*}
Set $C_T^k=(1/2)CM_k (l^{-\alpha_p}+(l+1)^{-\alpha_p})k^{\alpha_p}$ and $D_T^k=(1/2)DM_k (l^{-\alpha_p}+(l+1)^{-\alpha_p})k^{\alpha_p}$. Let $C'_T=\lim_{k \to \infty} C_T^k$  and $D'_T=\lim_{k \to \infty} D_T^k$. The limits exist since 
$\lim_{k \to \infty}M_k =1$ and $\lim_{k \to \infty}l/k=2Q'(2) / \dim T$ by the definition of $l$.

Let $C_T=C'_T - \epsilon$ and $D_T=D'_T + \epsilon$ for an arbitrary $\epsilon>0$.

Then

\begin{equation*}
    C_Tk^{-\alpha_p}(\dim T)^k <  \sum_{ n< \sqrt{k}\log k } \m_n(Q(x)^k)< D_T k^{-\alpha_p}(\dim T)^k,
\end{equation*}
so the result of the  case $\vert Q(-2)\vert \neq \dim T$ is proved.

\vspace{5mm}
\par
    Suppose $\vert Q(-2) \vert =\dim T$. Then $Q(x)$ is either an even or an odd polynomial. As the functions $\m_n$ are additive, by Proposition \ref{m_ninbeta} \begin{align*}
        \m_{n-1}(Q(x)^k)\equiv\begin{dcases}
            2 \sum_{\{j \mid B_n(\b_{n, j})=0 \text{ and $\b_{n, j} >0$}\}}c_{n, j }Q(b_{n, j})^k \quad \text{if $n-1  \equiv k \deg Q(x)$ (mod $2$)}\\
            0\qquad \qquad \qquad \qquad \quad \ \   \text{if $n-1 \not \equiv k \deg Q(x)$ (mod $2$)}.
        \end{dcases}
    \end{align*}
    (mod $o(k^{-2-\alpha} (\dim T)^k ))$.

    When taking $l$ in \eqref{choosel}, we choose $l$ so that the parity of $l$ is the same as the parity of the polynomial $Q(x)^k$, i.e. $l \equiv k$ (mod $2$) if $Q(x)$ is odd, and $l$ is even if $Q(x)$ is even. 
    
    Since $l$ and $Q(x)^k$ have the same parity, by \eqref{M_k}, $$\m_{n-1}(Q(x)^k)= (\dim T)^k 2^{-l}M_k \cdot \m_{n-1}(x^l) +o(k^{-1-\alpha}(\dim T)^k),$$ and the result follows.
\end{proof}

\begin{comment}

\begin{lemma}
    Given a tilting module $T$ and $Q\in \Z[x]$ such that $Q(V)=T$, suppose $\vert Q(-2) \vert = \dim T$. Then $$\m_{n-1}(Q(x)^k)= (\dim T)^k 2^{-l}M_k \cdot \m_{n-1}(x^l)$$ for some $l$, $M_k$ when $k$ is sufficiently large and $\lim_{k \to \infty}  M_k=1.$
\end{lemma}

\begin{proof}
    By Lemma \ref{lemma_p}, (4)
\end{proof}

\end{comment}

\printbibliography
\end{document}